\documentclass[10pt]{article}
\usepackage{amsmath,amsthm,amsfonts,latexsym,amsopn,verbatim,amscd,amssymb}
\theoremstyle{plain}
\newtheorem{theorem}{Theorem}[section]
\newtheorem{lemma}[theorem]{Lemma}
\newtheorem{corollary}[theorem]{Corollary}

\newtheorem{proposition}[theorem]{Proposition}

\newcommand{\bnum}{\begin{enumerate}}
\newcommand{\enum}{\end{enumerate}}

\numberwithin{equation}{section}

\DeclareMathOperator{\E}{E}

\begin{document}

\title{\textbf{On generalized commuting probability of finite rings}}
\author{Parama Dutta and Rajat Kanti Nath\footnote{Corresponding author}}
\date{}
\maketitle
\begin{center}\small{\it
Department of Mathematical Sciences, Tezpur University,\\ Napaam-784028, Sonitpur, Assam, India.\\



Emails:\, parama@gonitsora.com and rajatkantinath@yahoo.com}
\end{center}

\medskip

\begin{abstract}
Let $R$ be a finite ring and $r \in R$. The aim of this paper is to study  the probability that the commutator of a randomly chosen pair of elements of $R$ equals $r$.
\end{abstract}

\medskip



\medskip

\section{Introduction}
Let $F$ be an algebraic system having finite number of elements which is closed under a multiplication operation. The commuting probability of $F$, denoted by $\Pr(F)$, is the probability that a randomly chosen pair of elements of $F$ commute. That is,
\[
\Pr(F) = \frac{|\{(x, y) \in F \times F : xy = yx\}|}{|F \times F|}.
\]
The commuting probability of a finite algebraic system is originated from the works of Erd$\ddot{\rm o}$s and Tur$\acute {\rm a}$n in \cite{pEpT68}. In the last few decades, many mathematicians have studied $\Pr(F)$, including various generalizations,  considering $F$ to be a finite group (see \cite{Dnp13} and the references therein). In 1976,  MacHale \cite{dmachale} initiated the study of $\Pr(F)$ considering $F$ to be a finite ring. Somehow people did not pay much attention towards the study of commuting probability of finite rings and there generalizations over the years. At this moment, we have only the following papers   \cite{BM,BMS,jutireka,dmachale} on commuting probability of finite rings.

 Let $R$ be a finite ring and $r$  an element of $R$. In this paper, we consider the following ratio
\begin{equation}\label{mainformula}
{\Pr}_r(R) := \frac{|\lbrace(x, y)\in R\times R : [x, y] = r\rbrace|} {|R\times R|}
\end{equation}
where $[x, y] := xy - yx$ is the  commutator of $x$ and $y$. Notice that ${\Pr}_r(R)$ is  the probability that the commutator of a randomly chosen pair of elements of $R$ equals $r$. Also, ${\Pr}_r(R) = {\Pr}(R)$ if $r = 0$, the additive identity of $R$.

The motivation of this paper lies in \cite{nY15, PS08}, where analogous generalizations of commuting probability of finite groups are studied. Although there is no analogous concept of conjugacy class in ring theory, we obtain a computing formula for ${\Pr}_r(R)$. We also obtain some bounds and an invariance property of ${\Pr}_r(R)$ under isoclinism.





\section{Computing formulas and bounds}
Let $[R, R]$ and $[x, R]$ for $x \in R$  denote the additive subgroups of $(R, +)$ generated by the sets $\lbrace [x, y] : (x, y) \in R \times R \rbrace$ and $\lbrace [x, y] : y\in R\rbrace$ respectively. It can be seen that any element of $[x, R]$ is of the form $[x, y]$ for some $y \in R$. For any element $x \in R$, the centralizer of $x$ in $R$ is a subring given by $C_R(x) : = \{y \in R : xy = yx\}$. Also, the center of $R$ is given by $Z(R) := \underset{x \in R}{\cap}C_R(x)$. The following lemma gives a relation between $[x, R]$ and  $C_R(x)$.



\begin{lemma}\label{lemma1}
Let $R$ be a finite ring. If  $x \in R$ then
$
|[x, R]|=\frac {|R|}{|C_R(x)|}.
$
\end{lemma}

\begin{proof}
It is easy to see that the map  $[x, y] \mapsto y + C_R(x)$ defines an isomorphism from $[x, R]$ to $\frac {R}{C_R(x)}$. Hence, the lemma follows.
\end{proof}

For any two given elements $x$ and $r$ of $R$, we write $T_{x, r} = \{y\in R : [x, y] = r\}$. The following lemma gives a relation between $T_{x, r}$ and  $C_R(x)$.
\begin{lemma}\label{lemma2}
Let $R$ be a finite ring and  $x,r \in R$. Then we have the followings:
\begin{enumerate}
\item  $T_{x, r} \ne \phi$ if and only if $r \in [x, R]$.

\item If $T_{x, r}\neq \phi$ then $T_{x, r} = t + C_R(x)$ for some $t\in T_{x, r}$ and hence $|T_{x, r}| = |C_R(x)|$.
\end{enumerate}
\end{lemma}

\begin{proof}
Part (a) follows from the fact that $y \in T_{x, r}$ if and only if $r \in [x, R]$.
For part (b), let $t \in T_{x, r}$ and  $p \in t + C_R(x)$. Then $[x, p] = r$ and so $p \in T_{x, r}$. Therefore,  $t + C_R(x) \subseteq T_{x, r}$. Again, if $y \in T_{x, r}$ then  $(y - t) \in C_R(x)$ and so $y \in t + C_R(x)$. Therefore, $t + C_R(x)\subseteq T_{x, r}$. Hence, the result follows.
\end{proof}
\noindent Now we obtain the following computing formula for ${\Pr}_r(R)$ analogous to Lemma 3.1 of \cite{nY15}.

\begin{theorem}\label{com-thm}
Let $R$ be a finite non-commutative ring and $r \in R$. Then
\[
{\Pr}_r(R) = \frac {1}{{|R|}^2}\underset{r\in [x, R]}{\underset{x\in R}{\sum}}|C_R(x)| = \frac {1}{|R|}\underset{r\in [x, R]}{\underset{x\in R}{\sum}}\frac{1}{|[x, R]|}.
\]
\end{theorem}

\begin{proof}
Clearly $\{(x, y) \in R\times R : [x, y] = r\} = \underset{x\in R}{\cup}(\{x\}\times T_{x, r})$. Therefore, by \eqref{mainformula} and Lemma \ref{lemma2}, we have
\begin{equation} \label{comfor1}
{|R|}^2{\Pr}_r(R) = \underset{x \in R}{\sum} |T_{x, r}| = \underset{r\in [x, R]}{\underset{x\in R}{\sum}}|C_R(x)|.
\end{equation}
The second part follows  from \eqref{comfor1} and Lemma  \ref{lemma1}.
\end{proof}

\noindent As a corollary of Theorem \ref{com-thm}, we have
\begin{corollary}\label{formula1}
Let $R$ be a finite non-commutative ring. Then
\[
{\Pr}(R) = \frac {1}{{|R|}^2}\sum_{x\in R}|C_R(x)|  = \frac {1}{|R|}\sum_{x\in R}\frac{1}{|[x, R]|}.
\]
\end{corollary}

We also have the following results.
\begin{proposition}
Let $R_1$, $R_2$ be two finite non-commutative rings and $(r_1, r_2) \in R_1 \times R_2$.  Then
\[
{\Pr}_{(r_1, r_2)}(R_1 \times R_2) = {\Pr}_{r_1}(R_1){\Pr}_{r_2}(R_2).
\]
\end{proposition}

\begin{proof}
Let $X_i = \{(x_i, y_i) \in R_i\times R_i : [x_i, y_i] = r_i\}$ for $i = 1, 2$ and $Y = \{((x_1, x_2), (y_1, y_2)) \in (R_1\times R_2) \times (R_1\times R_2) : [(x_1, x_2),(y_1, y_2)]= (r_1, r_2)\}$. Then $((x_1, y_1), (x_2, y_2)) \mapsto ((x_1, x_2),(y_1, y_2))$ defines a bijective map from $X_1 \times X_2$ to $Y$. Therefore, $|Y| = |X_1||X_2|$ and hence the result follows.
\end{proof}

\begin{proposition}\label{symmetricity}
Let $R$ be a finite non-commutative ring and $r \in R$. Then
\[
{\Pr}_r(R) = {\Pr}_{-r}(R).
\]
\end{proposition}

\begin{proof}
Let $X = \{(x, y) \in R \times R : [x,y] = r\}$ and $Y = \{(y,x) \in R \times R : [y,x] = -r\}$. Then $(x,y) \mapsto (y,x)$ defines a bijective mapping from $X$ to $Y$. Therefore, $|X| = |Y|$ and the result follows.
\end{proof}


Now we deduce a relation between ${\Pr}_r(R)$ and
 the number of edges in the $r$-noncommuting graph of $R$. The $r$-noncommuting graph of a finite ring $R$, denoted by $\Gamma_{R}^{r}$, was introduced and studied in \cite{dn2017}. Recall that $\Gamma_{R}^{r}$ is a graph whose vertex set is $R$ and two distinct vertices $x$ and $y$ are adjacent if $[x, y] \ne r$ and $[y, x] \ne r$. Let $\E(\Gamma_{R}^{r})$ denote the set of edges of the graph $\Gamma_{R}^{r}$. Then we have the following result.


\begin{proposition}
Let $R$ be a finite non-commutative ring and $r \in R$.
\begin{enumerate}
\item If $r = 0$ then ${\Pr}_r(R) = 1 - \frac{2|\E(\Gamma_{R}^{r})|}{|R|^2}$.
\item If $r \neq 0$ then ${\Pr}_r(R) = \begin{cases}
1 - \frac{1}{|R|} - \frac{2|\E(\Gamma_{R}^{r})|}{|R|^2} & \text{ if } 2r = 0,\\
\frac {1}{2}\left(1 - \frac{1}{|R|} - \frac{2|\E(\Gamma_{R}^{r})|}{|R|^2}\right) & \text{ if } 2r \ne 0.
\end{cases}$
\end{enumerate}
\end{proposition}
\begin{proof}
Let $\mathcal{E} = \{(x, y) \in R \times R: x \ne y, [x,y] \neq r \text{ and } [y,x] \neq r\}$. Then $|\E(\Gamma_{R}^{r})| = \frac{|\mathcal{E}|}{2}$.
Consider the  sets $D = \{(x, y)\in R\times R: x = y\}$, $A = \{(x,y)\in R\times R:[x,y] = r\}$ and
$B = \{(x, y)\in R\times R:[y,x] = r\}$. Clearly, $\mathcal{E} \cap A \ne \phi$, $\mathcal{E} \cap B \ne \phi$  and  $R \times R = \mathcal{E} \cup D \cup A \cup B$. Also, $|D| = |R|, |A| = |R|^2{\Pr}_r(R)$ and $|B| = |R|^2{\Pr}_{-r}(R)$. Therefore, by Proposition \ref{symmetricity}, we have $|A| = |B|$. 

(a) If $r = 0$ then $D \subseteq A$ and $A = B$. Therefore, $\mathcal{E} = R \times R \setminus A$ and
\[
|\E(\Gamma_{R}^{r})| = \frac {{|R|}^2}{2}\left(1 - {\Pr}_r(R)\right).
\]
Hence, the result follows.

(b) Suppose that $r \ne 0$.
Then  $D \nsubseteq \mathcal{E}$, $A$ and $B$. Further, if $2r = 0$ then $A = B$. Therefore,  $\mathcal{E} = (R \times R \setminus D) \setminus A$ and
\[
|\E(\Gamma_{R}^{r})| = \frac {1}{2}\left({|R|}^2 - |R| - |R|^2{\Pr}_r(R)\right).
\]
If $2r \neq 0$ then $A \ne B$. Therefore, $\mathcal{E} = ((R \times R \setminus D) \setminus A) \setminus B$ and
\[
|\E(\Gamma_{R}^{r})| = \frac {1}{2}\left({|R|}^2 - |R| - 2|R|^2{\Pr}_r(R)\right).
\]
Hence, the result follows.
\end{proof}


 We conclude this section with the following bounds for ${\Pr}_r(R)$.
\begin{proposition}
Let $R$ be a finite non-commutative ring. If $r \ne 0$ then
\[
{\Pr}_r(R)\geq \frac{3}{|R : Z(R)|^2}.
\]

\end{proposition}

\begin{proof}
Since $r \ne 0$ we have the set ${\mathcal{C}} := \{(x, y) \in R \times R : [x, y] = r\}$ is non empty. Let $(s, k) \in {\mathcal{C}}$ then $(s, k) \notin Z(R)\times Z(R)$, otherwise $[s, k] = 0$. Since $(s, k) + \left(Z(R)\times Z(R)\right)$, $(s + k, k) + \left(Z(R)\times Z(R)\right)$ and $(s, k + s) + \left(Z(R)\times Z(R)\right)$ are three disjoint subsets of ${\mathcal{C}}$ we have $|{\mathcal{C}}| \geq 3|Z(R)|^2$. Hence, the result follows from \eqref{mainformula}.
\end{proof}

\begin{proposition}\label{ub02}
Let $R$ be a finite non-commutative ring. Then ${\Pr}_r(R) \leq {\Pr}(R)$ with equality if and only if $r = 0$. \end{proposition}
\begin{proof}
By Theorem \ref{com-thm} and Corollary \ref{formula1}, we have
\[
{\Pr}_r(R) = \frac {1}{{|R|}^2}\underset{r \in [x, R]}{\underset{x\in R}{\sum}}|C_R(x)|
 \leq \frac {1}{|{R|}^2}\underset{x\in R}{\sum}|C_R(x)|
=  \Pr(R).
\]
Clearly, the equality holds if and only if $r = 0$.
\end{proof}

\begin{proposition}
Let $R$ be a finite non-commutative ring. If $p$ the smallest prime dividing $|R|$ and $r\neq 0$ then
\[
{\Pr}_r(R)\leq \frac {|R| - |Z(R)|}{p|R|} < \frac {1}{p}.
\]
\end{proposition}
\begin{proof}
We have $R \ne Z(R)$. If $x \in Z(R)$ then $r \notin [x,R]$. If $x \in R \setminus Z(R)$ then $C_R(x) \ne R$. Therefore, by Lemma \ref{lemma1}, we have  $|[x, R]| = |R : C_r(x)| > 1$. Since   $p$ is the smallest prime dividing $|R|$ we have $|[x, R]| \geq p$.  Hence the result follows from Theorem \ref{com-thm}.
\end{proof}

\section{An invariance property }
In 1940, Hall \cite{pH40} introduced the concept of isoclinism between two groups. In 1995,
Lescot \cite{pL95}  showed that $\Pr(G_1) = \Pr(G_2)$ if $G_1$ and $G_2$ are  isoclinic finite groups. Various generalizations of this result can be found in \cite{DasNath2012, nath2011, nY2015, PS08}.
Recently, Buckley et. al. \cite{BMS} introduced the concept of $\mathbb{Z}$-isoclinism between two rings and showed that $\Pr(R_1) = \Pr(R_2)$ if $R_1$ and $R_2$ are two $\mathbb Z$-isoclinic finite rings (see \cite[Lemma 3.2]{BMS}). In this section, we generalize Lemma 3.2 of \cite{BMS}. Recall that two rings  $R_1$ and $R_2$ are said to be $\mathbb Z$-isoclinic if there exist additive group isomorphisms $\alpha :\frac {R_1}{Z(R_1)}\rightarrow \frac {R_2}{Z(R_2)}$   and $\beta :[R_1, R_1]\rightarrow [R_2, R_2]$ such that $\beta ([x_1, y_1])=[x_2, y_2]$ whenever  $\alpha (x_1 +  Z(R_1)) = x_2 +  Z(R_2)$ and $\alpha (y_1 +  Z(R_1)) = y_2 +  Z(R_2)$. The following theorem gives an invariance property of ${\Pr}_r(R)$ under $\mathbb Z$-isoclinism of rings.

\begin{theorem}
Let $R_1$ and $R_2$ be two finite non-commutative rings. If  $(\alpha,\beta)$ is a    $\mathbb Z$-isoclinism between $R_1$ and $R_2$ then
\[
{\Pr}_r(R_1) = {\Pr}_{\beta (r)}(R_2).
\]
\end{theorem}

\begin{proof}
By Theorem \ref{com-thm}, we have
\begin{align*}
{\Pr}_r(R_1)
=&\frac {|Z(R_1)|}{|R_1|}\underset{r \in [s_1, R_1]}{\underset{s_1 + Z(R_1)\in\frac {R_1}{Z(R_1)}}{\sum}}\frac {1}{|[s_1, R_1]|}
\end{align*}
noting that $r \in [s_1, R_1]$ if and only if $r \in [s_1 + z, R_1]$ for all $z \in Z(R_1)$.
Similarly, it can be seen that
\[
{\Pr}_{\beta (r)}(R_2) = \frac {|Z(R_2)|}{|R_2|}\underset{\beta (r)\in [s_2, R_2]}{\underset{s_2 + Z(R_2)\in\frac {R_2}{Z(R_2)}}\sum}
\frac {1}{|[s_2, R_2]|}.
\]
Since $(\alpha,\beta)$ is a    $\mathbb Z$-isoclinism between $R_1$ and $R_2$ we have  $\frac {|R_1|}{|Z(R_1)|} = \frac {|R_2|}{|Z(R_2)|}$, $|[s_1,R_1]| = |[s_2, R_2]|$ and $r \in [s_1, R_1]$ if and only if  $\beta (r) \in [s_2, R_2]$. Hence
\[
{\Pr}_r(R_1) = {\Pr}_{\beta (r)}(R_2).
\]
\end{proof}




\section*{Acknowledgment}
The first author would like to thank  DST for the INSPIRE Fellowship.

\end{document}